\documentclass[12pt]{amsart} 
\usepackage{amsmath}
\allowdisplaybreaks

\newtheorem{theorem}{Theorem}[section]
\newtheorem{corollary}[theorem]{Corollary}

\newtheorem{lemma}[theorem]{Lemma}
\newtheorem{proposition}[theorem]{Proposition}

\theoremstyle{definition}

\newtheorem{remark}[theorem]{Remark}

\def\Er{{\mathbb E}}

\def\Pr{{\mathbb P}}
\def\Qr{{\mathbb Q}}
\def\Rr{{\mathbb R}}

\def\Ac{{\mathcal{A}}}
\def\Bc{{\mathcal{B}}}

\def\Ec{{\mathcal{E}}}
\def\Fc{{\mathcal{F}}}

\def\Hc{{\mathcal{H}}}

\def\sign{\operatorname{sign}}

\def\({\left(}     
\def\){\right)}    
\def\[{\left[}     
\def\]{\right]}

\def\as{{\frenchspacing a.s.}~}
\def\trace{{\text { }Trace}}

\keywords{martingales}

\thanks{I thank Chitro Majumdar from RSRL (R-square RiskLab), UAE for valuable discussions and comments}

\begin{document}
\title[Martingales with independent increments]{Martingales with Independent Increments}
\author{Freddy Delbaen}
\address{Departement f\"ur Mathematik, ETH Z\"urich, R\"{a}mistrasse
101, 8092 Z\"{u}rich, Switzerland}
\address{Institut f\"ur Mathematik,
Universit\"at Z\"urich, Winterthurerstrasse 190,
8057 Z\"urich, Switzerland}
  \email{delbaen@math.ethz.ch}
\begin{abstract}
We show that a discrete time martingale with respect to a filtration with atomless innovations is the (infinite) sum of martingales with independent increments. For the continuous time filtration coming from  Brownian Motion filtration, we show that every finite dimensional $L^2$ martingale is the sum of a series of Gaussian martingales.
\end{abstract}

\maketitle
\section{$L^2$-Martingales}
We use standard probabilistic notation, $(\Omega,(\Fc_k)_{k\ge 0},{\mathbb P})$ is a probability space  equipped with a discrete time filtration. For convenience we suppose that $\Fc_0$ is the trivial sigma algebra. All norms -- except when otherwise stated, are meant to be the $L^2$ norm.  The norm on an $\Rr^m$ space is the Euclidean norm and is denoted by $|\, .\,  |$. The scalar product between elements $x,y\in\Rr^m$ is denoted by $x\cdot y$. We suppose that for each $k$ the innovation is sufficiently large to allow independent random variables that have a continuous distribution.  More precisely we suppose that for each $k\ge 1$, there is a $[0,1]$ uniformly distributed random variable $U_k$ that is independent of $\Fc_{k-1}$ and is $\Fc_k-$measurable.  It is shown in \cite{D3}, see also \cite {D2}, that this is equivalent to the property: for each $k$, $\Fc_k$ is atomless conditional to $\Fc_{k-1}$.  By $(X_k)_{k\ge 0}, X_0=0$ we denote an $\Rr^m-$valued $L^2$ martingale. In other words $X_k\in L^2(\Fc_k)$ and $\Er[X_k\mid\Fc_{k-1}]=X_{k-1}$. The aim of this short note is to prove
\begin{theorem} There is a sequence of martingales $Z^n=(Z_k^n)_{k\ge 1}$ such that for each $n$ we have $Z^n_0=0$ and such that $X_k=\sum_n Z^n_k$ where
\begin{enumerate} \item Each $Z^n_k-Z^n_{k-1}$ is independent of $\Fc_{k-1}$.
\item For each $k\ge 1$ the sum $\sum_n Z^n_k$ converges in $L^2$.
\item $\Vert X_k\Vert^2=\sum_n \Vert Z^n_k\Vert^2$.
\end{enumerate}
The martingales $Z^n$ have independent increments, for each $n$ the differences $(Z^n_k-Z^n_{k-1})_{k\ge 1}$ form an  independent system.
\end{theorem}

The basis of the proof is the following result we showed in \cite{D5}.
\begin{theorem} Suppose that $(\Omega,\Fc,{\mathbb P})$ is a probability space and that $\Ac$ is a sub sigma-algebra of $\Fc$.  Suppose that $\Fc$ is atomless conditionally to $\Ac$.  Let $\xi\in L^2(\Rr^m)$ satisfy $\Er[\xi\mid \Ac]=0$ and put $\xi_1=\xi$,  inductively $\eta_n\colon\Omega\rightarrow \Rr^m$ is independent of $\Ac$ and is the best $L^2$ approximation of $\xi_n$, i.e.
$$
\Vert \xi_n-\eta_n\Vert =\inf\{\Vert \xi_n-\zeta\Vert\mid \zeta\text{ is independent of }\Ac\},
$$
$\xi_{n+1}=\xi_n-\eta_n$. 
\begin{enumerate}
\item $\eta_n=\Er[\xi_n\mid \eta_n]$.
\item For each $n$: $\Vert \xi_1\Vert^2= \Vert \xi_{n+1}\Vert^2 +\Vert \eta_1\Vert^2+\ldots +\Vert \eta_n\Vert^2$\\
hence $\Vert \eta_1+\ldots+\eta_n\Vert\le \Vert \xi\Vert +\Vert \xi_{n+1}\Vert\le 2\Vert \xi\Vert$
\item $\Vert \eta_n\Vert \ge \frac{1}{2m}\Vert \xi_n\Vert_1$ (we need the $L^1-$norm).
\item $\xi_n\rightarrow 0$ in $L^2$, consequently $\xi_1=\sum_{n\ge 1} \eta_n$ in $L^2$ and $\Vert \xi\Vert^2=\sum_k\Vert \eta_k\Vert^2$.
\item For each $n$: $\xi_n=\sum_{k\ge n}\eta_k$ and $\Vert \xi_n\Vert^2=\sum_{k\ge n}\Vert \eta_k\Vert^2$.
\end{enumerate}
\end{theorem}
\noindent
We are now ready to prove the main theorem.
\begin{proof}  As shown in \cite{D5} for each $k$ there is a sequence $Y^n_k$ such that $Y^n_k$ is independent of $\Fc_{k-1}$, is $\Fc_k$ measurable, $\Er[Y^n_k\mid \Fc_{k-1}]=0$ and $X_k-X_{k-1}=\sum_n Y^n_k$.  The sum converges in $L^2$ and $\Vert X_k-X_{k-1}\Vert^2=\sum_n \Vert Y^n_k\Vert^2$.  We put
$$
Z^n_k=\sum_{s=1}^{s=k} Y^n_s.
$$
It is easily seen that each $Z^n$ defines an $L^2$ martingale.  For each $k$ we have $X_k=\sum_{s=1}^{s=k} (X_s-X_{s-1}) = \sum_{s=1}^{s=k} (\sum_n  Y^n_s)$. The  sums can be permuted and hence we get
$$
X_k=\sum_{s=1}^{s=k} (X_s-X_{s-1}) = \sum_n  \sum_{s=1}^{s=k} Y^n_s=\sum_n Z^n_k.
$$
For the $L^2$ norms we get:

\begin{eqnarray*}
\Vert X_k\Vert^2 &=&\sum_{s=1}^{s=k} \Vert (X_s-X_{s-1})\Vert^2 \\
&=&  \sum_{s=1}^{s=k}  \sum_n \Vert Y^n_s\Vert^2= \sum_n  \sum_{s=1}^{s=k} \Vert Y^n_s\Vert^2 =\sum_n \Vert Z^n_k\Vert^2 .
\end{eqnarray*}
That for each $n$, the random variables $Y^n_k$ form an independent system follows for instance from a calculation with characteristic functions. To see this let us fix $k$ and take $u_1,\ldots, u_k\in \Rr^m$. We will calculate
$$
\varphi(u_1,\ldots,u_k)=\Er\left[\exp\left(\sum_{s\le k} u_s\cdot Y^n_s\right)\right]
$$
by using successive conditional expectations. Clearly
\begin{eqnarray*}
\Er\left[\exp\left(\sum_{s\le k} u_s\cdot Y^n_s\right)\right]&=&\Er\left[\Er\left[\exp\left(\sum_{s\le k} u_s\cdot Y^n_s\right)\mid \Fc_{k-1}\right]\right]\\
&=&\Er\left[\exp\left(\sum_{s\le k-1} u_s\cdot Y^n_s\right)\right]\Er\left[\exp\left(u_k\cdot Y^n_k\right)\right]\\
& &\text{since $Y^n_k$ is independent of $\Fc_{k-1}$}\\
&=&\ldots\\
&=&\Er\left[\exp\left(u_1\cdot Y^n_1\right)\right]\ldots \Er\left[\exp\left(u_k\cdot Y^n_k\right)\right],
\end{eqnarray*}
proving independence.
\end{proof}

\section{Closed $L^2-$Martingales}

In this section we analyse the results of the previous section for closed martingales. We use the same hypothesis on the filtration $\Fc$ and we suppose that the $\Rr^m$ valued martingale $(X_k)_{k\ge 1}$ is bounded, i.e. $\sup_k \Vert X_k\Vert <\infty$.  In that case there is a random variable $X_\infty$ such that $X_k\rightarrow X_\infty$ in $L^2$ and almost surely.  Of course $X_k=\Er[ X_\infty\mid \Fc_k]$.  For simplicity and to avoid trivialities we again suppose that $X_0=\Er[X_\infty]=0$ and  $\Fc_0$ is the trivial sigma algebra. The random variables $Y^n_k, Z^n_k$ have the same meaning as in the previous section.

First we observe that the martingales $Z^n$ are all bounded in $L^2$.   This is immediate since
$$
\Vert Z^n_k\Vert^2=\sum_{s\le k}\Vert Y^n_s\Vert^2\le \sum_{s\le k}\Vert X_s-X_{s-1}\Vert^2= \Vert X_k\Vert^2.
$$
Each martingale $Z^n_k$ therefore converges in $L^2$ to a final value $Z^n_\infty$.
\begin{theorem} With the notation introduced above
$$
X_\infty =\sum_n Z^n_\infty,
$$
in $L^2$ and $\Vert X_\infty\Vert^2=\sum_n\Vert Z^n_\infty\Vert^2$.
\end{theorem}
\begin{proof}   We start by proving the equality $\Vert X_\infty\Vert^2=\sum_n\Vert Z^n_\infty\Vert^2$.
\begin{eqnarray*}
\Vert X_\infty\Vert^2 &=& \sum_{k\ge 1} \Vert X_k-X_{k-1}\Vert^2\\
&=& \sum_{k\ge 1} \sum_n \Vert Y^n_k\Vert^2\\
&=& \sum_{n} \sum_{k\ge 1} \Vert Y^n_k\Vert^2\\
&=& \sum_k \Vert Z_\infty^n\Vert^2.
\end{eqnarray*}
To prove convergence we proceed in the usual way.  We take $\varepsilon>0$. From the convergence $X_k\rightarrow X_\infty$ in $L^2$, we deduce that there is $k_0$ such that for  all $k\ge k_0$:
$$
\Vert X_\infty-X_k\Vert^2 \le \varepsilon^2
$$
Hence also for all $k\ge k_0$ and all $N$:
\begin{eqnarray*}
\Vert Z^1_\infty&+&\dots+Z^N_\infty-(Z^1_k+\dots+Z^N_k)\Vert^2\\
&=& \sum_{s>k}\Vert Y_s^1+\ldots+Y_s^N\Vert^2\\
&\le&4\sum_{s>k}\Vert X_s-X_{s-1}\Vert^2\\
&\le&4\,\Vert X_\infty-X_k\Vert^2\le 4\varepsilon^2.
\end{eqnarray*}
Now we choose $N_0$ such that 
$$
\Vert X_{k_0}-(Z^1_{k_0}+\ldots+Z^N_{k_0})\Vert\le \varepsilon,
$$
for all $N\ge N_0$. The usual splitting then gives for all $N\ge N_0$:
\begin{eqnarray*}
\Vert X_\infty&-&(Z^1_\infty+\ldots+Z^N_\infty\Vert\\
&\le&\Vert X_\infty -X_{k_0}\Vert +\Vert X_{k_0}-\left(Z^1_{k_0}+\ldots+Z^N_{k_0}\right) \Vert\\
& &\quad\quad+ \Vert\left(Z^1_{k_0}+\ldots+Z^N_{k_0}\right)-\left(Z^1_{\infty}+\ldots+Z^N_{\infty}\Vert \right)\\
&\le& 4\varepsilon.
\end{eqnarray*}
\end{proof}

\section{Continuous Time Martingales}

Before we discuss the approximation of  martingales, we first prove a lemma that will serve later on.  The notation is the following $(E,\Ec,\mu)$ is a probability space and $\Hc\subset \Ec$ is a sub sigma-algebra. Expectations in this probability space are denoted by $\Er_\mu$. For $x\in\Rr^m$ we define $\sign(x)=\frac{x}{|x|}$ if $x\neq 0$ and $\sign(0)=(1,0,\ldots,0)$.

\begin{lemma} Let $\xi_1\in L^2(\mu)$ be a square integrable function $\xi\colon E\rightarrow \Rr^m$ and define inductively
$$
\xi_{n+1}=\sign(\xi_n)\(|\xi_n|-\Er_\mu[|\xi_n|\mid\Hc]\)=\xi_n-\sign(\xi_n)\Er_\mu[|\xi_n|\mid\Hc].
$$
We have $\xi_n\rightarrow 0$ in $L^2$. Consequently $\xi_1=\sum_{n\ge 1}\sign(\xi_n)\Er[|\xi_n|\mid \Hc]$ where the sum converges in $L^2$ and $\Vert \xi_1\Vert^2=\sum_n \Vert \Er[ |\xi_n|\mid \Hc]\Vert^2$.
\end{lemma}
\begin{proof} The properties of conditional expectation show that
$$ \Vert \xi_n\Vert^2=
\Vert \Er_\mu [|\xi_n|\mid \Hc]\Vert^2 + \Vert \xi_{n+1}\Vert^2. $$
We can telescope this to yield
$$
\Vert \xi_1\Vert^2=\sum_{k=1}^{k=n} \Vert \Er_\mu[|\xi_k|\mid \Hc]\Vert^2 +\Vert \xi_{n+1}\Vert^2.
$$
This implies $\sum_{k\ge 1} \Vert \Er_\mu[|\xi_k|\mid \Hc]\Vert^2 <\infty$.  From this we get that $\Er_\mu[|\xi_k|\mid \Hc]\rightarrow 0$ in $L^2$ and hence also $\Er_\mu[|\xi_k|]\rightarrow 0$.
Furthermore $|\xi_{n+1}|=|\,|\xi_n|-\Er_\mu[|\xi_n|\mid\Hc]\,|\le \max(|\xi_n|,\Er_\mu[|\xi_n|\mid\Hc])$ which in turn implies
$$
|\xi_{n+1}|^2\le \max(|\xi_n|^2,\Er_\mu[|\xi_n|\mid\Hc]^2)\le |\xi_n|^2 + \Er_\mu[|\xi_n|\mid \Hc]^2.
$$

Telescoping this inequality yields the inequality:
$$
|\xi_{n+1}|^2 \le |\xi_1|^2 + \sum_{k=1}^{k=n}\Er_\mu[|\xi_k|\mid \Hc]^2\le  |\xi_1|^2 + \sum_{k\ge 1}\Er_\mu[|\xi_k|\mid \Hc]^2.
$$
However $\sum_{k\ge 1} \Vert \Er_\mu[|\xi_k|\mid \Hc]\Vert^2 <\infty$ and hence the sequence $|\xi_n|^2$ is uniformly integrable. Since the sequence $(\xi_n)_n$ already converges to $0$ in $L^1$ we get that
it also converges to $0$ in $L^2$. The expression
$$
\Vert \xi_1\Vert^2=\sum_{k=1}^{k=n} \Vert \Er_\mu[|\xi_k|\mid \Hc]\Vert^2 +\Vert \xi_{n+1}\Vert^2.
$$
and $\Vert \xi_{n+1}\Vert\rightarrow 0$ now complete the proof of the last line of the lemma.
\end{proof}

We can now extend the results of the previous sections  to the case of the  one  dimensional Brownian filtration $\Fc$.  To avoid normalising factors and extra time transforms, we restrict the time interval to $[0,1]$. We first recall the structure of martingales with independent increments.
\begin{proposition}  Suppose that $Y$ is a one dimensional $L^2$ martingale so that for each $0\le t<s\le 1$, $Y_s-Y_t$ is independent of $\Fc_t$. In this case there is a deterministic function $0\le f \in L^2[0,1]$ as well as a predictable function $\varphi$, satisfying $|\varphi |=1$ \as on the product space $[0,1]\times \Omega$ such that $dY_t=\varphi_t\, f(t)\,dW_t$. Conversely if the one dimensional martingale $Y$ satisfies $dY_t=\varphi_t\, f(t)\,dW_t$ with a deterministic function $0\le f \in L^2[0,1]$ as well as a predictable function $\varphi$, satisfying $|\varphi |=1$ \as on the product space $[0,1]\times \Omega$, then $Y$ has independent increments and is a Gaussian process.
\end{proposition}
\begin{proof} Without loss of generality we may suppose that $Y_0=0$. Suppose that $dY_t=H_t\,dW_t$ where $H$ is predictable. The Kunita-Watanabe equality then shows that for $0\le f$ and $f(t)^2=\Er[H_t^2]$ we have $\Er[Y_t^2]=\int_0^t f(u)^2\,du$. Because the increments are independent of the past, it is obvious that for $t<s\le1$ and for each $n$, $$\sum_{k=0}^{k=2^n-1}\(Y_{t +(s-t)(k+1)/2^n}-Y_{t +(s-t)(k)/2^n}\)^2$$ is independent of $\Fc_t$. Since these sums converge to $\langle Y,Y\rangle_s-\langle Y,Y\rangle_t$, \as we find that these differences are independent of $\Fc_t$. This implies that 
$\int_0^tH_u^2\,du- \int_0^t \Er[H_u^2]^2\,du=\langle Y,Y\rangle_t- \int_0^t f(u)^2\,du$ is a martingale in a Brownian filtration. This is only possible if it is constantly equal to $0$. This in turn implies that \as on $[0,1]\times \Omega$, $H=\varphi f$ with $\varphi$ predictable and $|\varphi |=1$.  The converse is an obvious calculation using characteristic functions and Ito's formula.
\end{proof}
\begin{remark} The previous result is probably  part of exercises in Brownian Motion theory. The author could not find references and therefore included a proof. The reader can consult the paper by Millar, \cite{Millar}, where besides convergence of the quadratic variation also references are given to earlier results, for instance by Doob.  However these results mention the representation without using the predictable process $\varphi$ and use an alternative Brownian Motion.  The exercises in Revuz-Yor, \cite{RY}, exercise 1.14, page 186 and exercise 1.35, page 133 point in the same direction as the proposition above. The last part of the proposition can already be found in Doob's book, \cite{Doob}, Theorem 5.3, page 449.
\end{remark}
We are now ready to state and prove the main result of this section.
\begin{theorem} Let $X$, $X_0=0$, $0\le t\le 1$ be an $L^2$ martingale with respect to the filtration generated by $1-$dimensional Brownian Motion $W$, then there exists a sequence of Gaussian martingales $Y^n$ of the form $dY^n_t=\varphi_n(t)f_n(t)\,dW_t$ where each $\varphi_n$ is predictable, $|\varphi_n|=1$ and each $f_n\in L^2[0,1]$ is deterministic.  The martingale $X$ is the sum $\sum_n Y^n$ where the sum converges in $L^2$.  Furthermore $\Vert X_1\Vert^2=\sum_n\Vert Y^n_1\Vert^2$.
\end{theorem}
\begin{proof} We represent $X$ by its stochastic integral $dX_t= H_t\,dW_t$ and we regard $H$ as an element of $(E,\Ec,\mu)$ where $E=[0,1]\times \Omega$, $\Ec$ is the sigma algebra of predictable events and $\mu$ is the product measure $m\otimes\Pr$ where $m$ is the Lebesgue measure. The sigma algebra $\Hc$ is the sigma algebra $\Bc\otimes\{\emptyset,\Omega\}$ enlarged with the evanescent sets and where $\Bc$ is as expected the Borel sigma algebra on $[0,1]$. Now we can apply the lemma above and get $H=\sum_n \varphi_nf_n$, the sum being convergent in $L^2\([0,1]\times \Omega\)$ and $\int_{[0,1]\times\Omega} H^2=\sum_n\Vert f_n\Vert^2$. If we take for $Y^n$ the stochastic integral $\int \varphi_n f_n\,dW$ we get the desired sequence.
\end{proof}
\begin{corollary} Let $(\Omega,\Fc,\Pr)$ be an atomless probability space and let $\xi\in L^2$.  In case $\Er[\xi]=0$ there is a sequence of standard normal variables $\psi_n$ and real numbers $a_n$ such that $\xi=\sum_n a_n\psi_n$ where the sum converges in $L^2$ and $\Vert \xi \Vert^2=\sum_n a_n^2$.
\end{corollary}
\begin{proof}  We sketch a ``sledge hammer" proof. Since $\Omega$ is atomless there is a mapping $\alpha\colon \Omega\rightarrow [0,1]$ such that the image measure is the Lebesgue measure and such that $\xi$ is measurable for the atomless sigma algebra generated by $\alpha$. The spaces $[0,1]$ with the Lebesgue measure on the Borel sets and the space $C[0,1]$ with the Borel sigma algebra and the Wiener measure, are isomorphic as probability spaces.  The composition of $\alpha$ and this isomorphism allows to construct a Brownian Motion on $\Omega$ such that $\xi\in L^2(\Fc_1)$.  We can now apply the result of the theorem to the martingale $X_t=\Er[\xi\mid \Fc_t]$ and get the corresponding Gaussian martingales $Y^n$. The series $\sum_n Y^n_1$ converges to $\xi$ in $L^2$ and $\Vert \xi\Vert^2=\sum_n\Vert Y^n_1\Vert^2$.  We now put $\psi_n=\frac{Y^n_1}{\Vert Y^n_1\Vert}$ and $ a_n=\Vert Y^n_1\Vert$.
\end{proof}

\section{A financial Interpretation}

In mathematical finance the gains process $X$ of a portfolio is (under good boundedness conditions) a martingale with respect to a risk neutral measure $\Qr$. In the Samuelson-Black-Scholes model the driving force is a Brownian Motion and the market is complete.  That means that the gains process of the stock allows to represent every martingale and hence also the driving Brownian Motion.  As a result a gains process $X$ of a portfolio can be represented as an $L^2$ sum of gains processes that are Gaussian processes under the risk neutral measure $\Qr$.  This may sound strange since many gains processes are bounded below whereas Gaussian processes are not. Also Gaussian processes are symmetric whereas gains processes in general are not.  The Gaussian processes in the series expansion are not orthogonal and certainly not independent.  The series of such processes  converge in $L^2$ and it is perfectly possible that the partial sums are not in $L^\infty$ but their limit is a bounded random variable. There is a  difference between  this expansion and the results obtained by Carr, Geman, Madan and Yor, \cite{1}. 

\section{Continuous time Martingales in More Dimensions}

This section is an attempt to generalise the results for 1-dimension to more dimensions. Since it uses maybe less known features from linear algebra, we decided to put it in a separate section.  We start with a proposition (without proof) that summarises the topics needed to replicate the ideas of the 1-dimensional case.  The measurability statements can be proved using explicit constructions of the polar decomposition or constructions of the related singular value decomposition. We agree that the details are technical and maybe not available in standard textbooks but including them would overload the paper with non-essential paragraphs.
\begin{proposition}
\begin{enumerate}
\item For each $n\times n$ matrix, $A$, we can find a symmetric positive semi-definite matrix $R$ and an orthogonal matrix $O$ such that $A=RO$.  In case the rank of $A$ equals $n$, the decomposition is unique. The matrix $R$ is the square root of $A\,A^*$, i.e. $A\,A^*=R^2$. We put $R=|A|$ to simplify notation. This decomposition is called the polar decomposition of $A$.\footnote{ There is also a polar decomposition written in the reverse order $A=U T$, where $U$ is orthogonal and $T$ is symmetric positive semi-definite.}
\item There are Borel measurable mappings $\omega,\rho\colon \Rr^{n^2}\rightarrow \Rr^{n^2}\times \Rr^{n^2}$ mapping $A$ to its polar decomposition $A=\rho(A)\omega(A)$.
\item For an $n\times n$ matrix $A=(a_{i,j})_{1\le i\le n,1\le j\le n}$ the Hilbert Schmidt norm of $A$ is $\Vert A\Vert^2=\sum_{i,j} a_{i,j}^2$. If $O$ is orthogonal then $\Vert OA\Vert =\Vert A\Vert=\Vert AO\Vert=\Vert O^*A\,O\Vert$.
\item If $R_1,R_2$ are two symmetric positive semi-definite $n\times n$ matrices then $\Vert R_1-R_2\Vert^2\le \Vert R_1\Vert^2 +\Vert R_2\Vert^2$.
\end{enumerate}
\end{proposition}
\begin{proof}
We will only prove the statement on the norm inequality. Using orthogonal matrices we can diagonalise $R_1$ to get a diagonal matrix $D= O^*\,R_1\,O$. This operation will not necessarily diagonalise $R_2$. But it will not change the Hilbert Schmidt norms. The matrix $B=O^*\,R_2\,O=(b_{i,j})_{i,j}$ is still positive semi-definite and hence has nonnegative elements on the diagonal. Let us now calculate the norm of $D-B$. This gives
$$
\Vert D-B\Vert^2=\sum_i(d_{i,i}-b_{i,i})^2+\sum_{i\neq j} b_{i,j}^2.
$$
Because $d_{i,i},b_{i,i}\ge 0$ we have $|d_{i,i}-b_{i,i}|\le \max(d_{i,i},b_{i,i})$.  Hence we get
\begin{eqnarray*}
\Vert R_1-R_2\Vert^2&=&\Vert D-B\Vert^2\\
 &\le&\sum_i  \max(d_{i,i},b_{i,i})^2+\sum_{i\neq j} b_{i,j}^2\\
 &\le& \sum_i (d_{i,i}^2+b_{i,i}^2)+\sum_{i\neq j} b_{i,j}^2\\
 &\le& \sum_i d_{i,i}^2 +\sum_{i,j}b_{i,j}^2\\
 &\le&\Vert D\Vert ^2+\Vert B\Vert^2=\Vert R_1\Vert^2 +\Vert R_2\Vert^2.
\end{eqnarray*}
\end{proof}
\begin{remark} If $A=RO$ is the polar decomposition then $\Vert A\Vert^2= \trace(A\,A^*)=\trace(R^2)=\Vert R\Vert^2$, or written differently $\Vert A\Vert =\Vert\,|A|\,\Vert$.
\end{remark}
Using the above proposition we can now prove the following. We again use the notation: $(E,\Ec,\mu)$ is a probability space with expectation operator $\Er_\mu$ and $\Hc\subset \Ec$ is a sub sigma-algebra.
\begin{lemma} Let $\xi_1\in L^2(\mu)$ be a square integrable function taking values in the set of $n\times n$ matrices. We inductively define 
$$
\xi_{n+1}=\(|\xi_n|-\Er_\mu[|\xi_n|\mid\Hc]\) \omega(\xi_n)=\xi_n-\Er_\mu[|\xi_n|\mid\Hc]\omega(\xi_n),
$$
where $\omega(\xi_n)$ is the orthogonal matrix in the polar decomposition of $\xi_n$. We have $\xi_n\rightarrow 0$ in $L^2$. Consequently $\xi_1=\sum_{n\ge 1}\Er[|\xi_n|\mid \Hc]\omega(\xi_n)$ where the sum converges in $L^2$ and $\int \Vert \xi_1\Vert^2\,d\mu=\sum_n \int \Vert \Er[ |\xi_n|\mid \Hc]\Vert^2\,d\mu$.
\end{lemma}
\begin{proof} The properties of conditional expectation and the properties of the Hilbert-Schmidt norm show that
$$\int \Vert \xi_n\Vert^2\,d\mu= \int\Vert\,| \xi_n|\,\Vert^2\,d\mu=
\int \Vert \Er_\mu [|\xi_n|\mid \Hc]\Vert^2\,d\mu + \int\Vert \xi_{n+1}\Vert^2\,d\mu. $$
We can telescope this to yield
$$
\int\Vert \xi_1\Vert^2\,d\mu=\sum_{k=1}^{k=n} \int\Vert \Er_\mu[|\xi_k|\mid \Hc]\Vert^2\,d\mu +\int\Vert \xi_{n+1}\Vert^2\,d\mu.
$$
This implies $\sum_{k\ge 1} \int\Vert \Er_\mu[|\xi_k|\mid \Hc]\Vert^2\,d\mu <\infty$.  From this we get that $\Er_\mu[|\xi_k|\mid \Hc]\rightarrow 0$ in $L^2$ and hence also $\Er_\mu[|\xi_k|]\rightarrow 0$.
As we have shown above 
$$
\Vert \xi_{n+1}\Vert^2\le \Vert \xi_n\Vert^2 + \Vert \Er_\mu[|\xi_n|\mid \Hc]\Vert^2.
$$

Telescoping this inequality yields the inequality:
$$
\Vert\xi_{n+1}\Vert^2 \le \Vert \xi_1\Vert^2 + \sum_{k=1}^{k=n}\Vert \Er_\mu[|\xi_k|\mid \Hc]\Vert^2\le \Vert \xi_1\Vert^2 + \sum_{k\ge 1}\Vert \Er_\mu[|\xi_k|\mid \Hc]\Vert^2.
$$
However $\sum_{k\ge 1} \int\Vert \Er_\mu[|\xi_k|\mid \Hc]\Vert^2\,d\mu <\infty$ and hence the sequence $\Vert\xi_n\Vert^2$ is uniformly integrable. Since the sequence $(\xi_n)_n$ already converges to $0$ in $L^1$ we get that
it also converges to $0$ in $L^2$. The expression
$$
\int \Vert \xi_1\Vert^2\,d\mu=\sum_{k=1}^{k=n}\int  \Vert \Er_\mu[|\xi_k|\mid \Hc]\Vert^2\,d\mu +\int \Vert \xi_{n+1}\Vert^2\,d\mu.
$$
and $\int \Vert \xi_{n+1}\Vert^2\,d\mu\rightarrow 0$ now form the proof of the last line of the lemma.
\end{proof}
We are now ready to state and prove that the main result of section 3 also holds for higher dimensions.
\begin{theorem} Let $X$, $X_0=0$, $0\le t\le 1$ be a $d-$dimensional $L^2$ martingale with respect to the filtration generated by the $d-$dimensional Brownian Motion $B$, then there exists a sequence of Gaussian martingales $Y^n$ of the form $dY^n_t=f_n(t)\,O_n(t)\,dB_t$ where each $O_n$ is a predictable process taking values in the set of orthogonal matrices and each $f_n\in L^2[0,1]$ is  deterministic taking values in the set of symmetric positive semi-definite matrices.  The martingale $X$ equals the sum $\sum_n Y^n$ where the sum converges in $L^2$.  Furthermore $\Vert X_1\Vert^2=\sum_n\Vert Y^n_1\Vert^2$. The process defined by $O_n(t)\,dB_t$ is a $d-$dimensional Brownian Motion.
\end{theorem}
\begin{proof} The proof is a repetition of the proof of the main result of section 3. Let $B$ be defined on $\Omega$ endowed with the usual filtration $(\Fc_t)_t$ generated by $B$ and equipped with the probability $\Pr$. The space $E=[0,1]\times \Omega$ is endowed with the restriction of the product measure $d\mu=dt\times d\Pr$ to the sigma algebra of predictable events, denoted here by $\Ec$. The sigma algebra $\Hc$ is generated by the mapping $(t,u)\rightarrow t$. By the Kunita-Watanabe representation theorem we can represent the martingale $X$ by its stochastic integral: $dX_t= A_t\,dB_t$ where $A$ is a predictable process taking values in the space of $n\times n$ matrices. We have $\Vert X\Vert^2=\int_{[0,1]\times \Omega}d\mu \Vert A\Vert^2$. We  can now apply the lemma above and get $A=\sum_n f_n\,O_n$, $f_n(t)$ are deterministic and symmetric positive semi-definite and $O_n$ are predictable taking values in the set of orthogonal matrices. The sum is convergent in $L^2\([0,1]\times \Omega\)$ and $\int_{[0,1]\times\Omega} \Vert A\Vert^2=\sum_n\int \Vert f_n\Vert^2$. If we take for $Y^n$ the stochastic integral $\int f_n\,O_n\,dB$ we get the desired sequence.
\end{proof}
In the same way as for the one dimensional case one can prove
\begin{corollary} Let $(\Omega,\Fc,\Pr)$ be an atomless probability space and let $\xi\in L^2$ be a $d-$dimensional random variable.  In case $\Er[\xi]=0$ there is a sequence of $d-$dimensional Gaussian variables $\psi_n$  such that $\xi=\sum_n \psi_n$ where the sum converges in $L^2$ and $\Vert \xi \Vert^2=\sum_n \Vert \psi_n\Vert^2$.
\end{corollary}

\end{document}